\documentclass[reqno,11pt]{amsart}
\usepackage{a4wide,color,eucal,enumerate,mathrsfs}
\usepackage[normalem]{ulem}
\usepackage{amsmath,amssymb,epsfig,amsthm} 
\usepackage[latin1]{inputenc}
\usepackage{graphicx}
\usepackage{psfrag}
\numberwithin{equation}{section}

\newtheorem{theorem}{Theorem}[section]

\newtheorem{corollary}[theorem]{Corollary}
\newtheorem{lemma}[theorem]{Lemma}

\newtheorem{question}[theorem]{Question}

\theoremstyle{definition}

\theoremstyle{remark}

\newcommand{\R}{\mathbb{R}}

\def\rr{{\mathbb R}}
\def\rn{{{\rr}^n}}

\def\fz{\infty}

\def\boz{{\Omega}}

\def\bint{{\ifinner\rlap{\bf\kern.35em--}
\int\else\rlap{\bf\kern.45em--}\int\fi}\ignorespaces}

\def\bbint{{\ifinner\rlap{\bf\kern.35em--}
\hspace{0.078cm}\int\else\rlap{\bf\kern.45em--}\int\fi}\ignorespaces}

\def\r{\right}
\def\lf{\left}

\def\bint{{\ifinner\rlap{\bf\kern.35em--}
\int\else\rlap{\bf\kern.45em--}\int\fi}\ignorespaces}

\begin{document}

\title[Bi-Lipschitz invariance of planar $BV$- and $W^{1,1}$-extension domains]{Bi-Lipschitz invariance of\\ planar $BV$- and $W^{1,1}$-extension domains}

\author{Miguel Garc\'ia-Bravo}
\author{Tapio Rajala}
\author{Zheng Zhu}

\address{University of Jyvaskyla \\
         Department of Mathematics and Statistics \\
         P.O. Box 35 (MaD) \\
         FI-40014 University of Jyvaskyla \\
         Finland}
         
\email{miguel.m.garcia-bravo@jyu.fi}         
\email{tapio.m.rajala@jyu.fi}
\email{zheng.z.zhu@jyu.fi}

\thanks{The first two authors acknowledge the support from the Academy of Finland, grant no.~314789. The third author was supported by the Academy of Finland via the Centre of Excellence in Analysis and Dynamics Research, grant no.~323960.}
\date{\today}


\begin{abstract}
We prove that a bi-Lipschitz image of a planar $BV$-extension domain is also a $BV$-extension domain, and that a bi-Lipschitz image of a planar $W^{1,1}$-extension domain is again a $W^{1,1}$-extension domain.
\end{abstract}


\maketitle


\section{Introduction}

Let $\boz\subset\rr^n$ be a domain. For $1\leq p\leq\fz$, we define the Sobolev space $W^{1, p}(\boz)$ by setting 
\[W^{1, p}:=\{u\in L^p(\boz)\,:\, \nabla u\in L^p(\boz;\rn)\},\]
where $\nabla u$ means the weak (distributional) derivative of $u$. The Sobolev space $W^{1, p}(\boz)$ is equipped with the norm 
\[\|u\|_{W^{1, p}(\boz)}:=\|u\|_{L^p(\boz)}+\||\nabla u|\|_{L^p(\boz)}.\]
We say that $\boz$ is a $W^{1, p}$-extension domain if there exists a bounded extension operator $T\colon W^{1, p}(\boz)\to W^{1, p}(\rn)$, meaning that for every $u \in W^{1,p}(\boz)$ we have $T(u)\big|_\boz\equiv u$ and 
\[\|T(u)\|_{W^{1, p}(\rn)}\leq C\|u\|_{W^{1, p}(\boz)},\]
where the constant $C$ is independent of $u$. The minimal possible constant $C$ above is denoted by $\|T\|$. Based on results in \cite{HKT:JFA, HKT:Rev}, for $1<p\leq\fz$, whenever $\boz$ is a $W^{1,p}$-extension domain, we can construct a bounded linear extension operator $T\colon W^{1, p}(\boz)\to W^{1, p}(\rn)$. By \cite{KRZ2}, for planar bounded simply connected $W^{1, 1}$-extension domains $\boz$, we can also construct a bounded linear extension operator $T\colon W^{1, 1}(\boz)\to W^{1, 1}(\rr^2)$. The classical results due to Calder\'on and Stein \cite{calderon, stein} tell us that Lipschitz domains are $W^{1, p}$-extension domains, for every $1\leq p\leq\fz$. Later, Jones \cite{Jones} defined a class of so-called $(\epsilon, \delta)$-domains which are a generalization of Lipschitz domains. He also proved that these domains are $W^{1, p}$-extension domains for every $1\leq p\leq\fz$. Moreover, in the works \cite{KRZ1, KRZ2, shvar:JFA}, a geometric characterization of planar bounded simply connected $W^{1, p}$-extension domains was established.

For arbitrary $u\in W^{1, p}(\rn)$ with $1\leq p\leq\fz$, the inequality
\[|u(x)-u(y)|\leq |x-y|\lf(CM[|\nabla u|](x)+CM[|\nabla u|](y)\r)\]
holds on every Lebesgue point of $u$, where the constant $C$ is independent of $u$ and $M[|\nabla u|]$ denotes the Hardy-Littlewood maximal function of $|\nabla u|$. Motivated by this estimate, Haj\l{}asz defined the so-called Haj\l{}asz-Sobolev space $M^{1, p}(\boz)$ which consists of all functions $u\in L^p(\boz)$ such that there exists a function $0\leq g\in L^p(\boz)$ satisfying the inequality
\begin{equation}\label{eq:Hajlasz}
|u(x)-u(y)|\leq |x-y|(g(x)+g(y))
\end{equation}
for every $x, y\in\boz\setminus F$ where the exceptional set $F$ satisfies $|F|=0$. We use $\mathcal D_p(u)$ to denote the class of all nonnegative functions $g\in L^p(\boz)$ which satisfy the inequality \eqref{eq:Hajlasz}. The Haj\l{}asz-Sobolev space $M^{1, p}(\boz)$ is then equipped with the norm 
\[\|u\|_{M^{1, p}(\boz)}:=\|u\|_{L^p(\boz)}+\inf_{g\in\mathcal D_p(u)}\|g\|_{L^p(\boz)}.\]
For $1\leq p\leq\fz$, one always has $M^{1, p}(\boz)\subset W^{1, p}(\boz)$ and the inclusion is strict for $p=1$, see \cite{ks}.
By \cite{hajlasz}, for $1<p\leq\fz$, the equality $M^{1, p}(\boz)=W^{1, p}(\boz)$ holds for a bounded $W^{1, p}$-extension domain $\boz$. Similarly to $W^{1, p}$-extension domains, a domain $\boz\subset\rn$ is said to be an $M^{1, p}$-extension domain, if there exists a bounded extension operator $T\colon M^{1, p}(\boz)\to M^{1, p}(\rn)$ (the existence of such an operator implies the existence of a linear one for the cases $1\leq p<\infty $ by \cite{HKT:Rev}). Observe that since $M^{1, \fz}(\boz)$ consists of Lipschitz functions, by Kirszbraun theorem, every domain is a $M^{1, \fz}$-extension domain. Furthermore, in \cite{HKT:Rev},  it is proved that for $1\leq p<\fz$ we have that $\boz$ is an $M^{1,p}$-extension domain if and only if $\boz$ is Ahlfors $n$-regular. We say that a domain $\boz\subset\rn$ is Ahlfors $n$-regular, if for every $x\in\boz$ and $0<r<1$, we have 
\[|B(x, r)\cap\boz|\geq c|B(x, r)|\]
with a constant $0<c<1$ independent of $x$ and $r$. Combining these two results, one can prove that for $1<p<\fz$, a domain $\boz\subset\rn$ is a $W^{1, p}$-extension domain if and only if $W^{1, p}(\boz)=M^{1, p}(\boz)$ and $\boz$ is Ahlfors $n$-regular (see \cite[Theorem 5]{HKT:JFA}). 

In the case that $\boz$ and $\boz'$ are bi-Lipschitz equivalent, one can easily check that $M^{1, p}(\boz)=M^{1, p}(\boz')$ isomorphically, and also the fact that $\boz$ is Ahlfors $n$-regular if and only if $\boz'$ is Ahlfors $n$-regular. Moreover, a domain is a $W^{1, \fz}$-extension domain if and only if it is locally quasiconvex and local quasiconvexity is bi-Lipschitz invariant. We say that a set $E\subset\rr^n$ is quasiconvex, if there exists a  constant $c\ge 1$ such that for every $x, y\in E$, there exists a rectifiable curve $\gamma\subset E$ connecting $x$ and $y$ with the length controlled from above by $c|x-y|$. In such case $E$ is called $c$-quasiconvex. These last observations lead to the following theorem, presented by Haj\l{}asz, Koskela and Tuominen in \cite{HKT:JFA}.
\begin{theorem}\label{thm:HKT}
If $\boz$ and $\boz'$ are bi-Lipschitz equivalent, for $1<p\leq\fz$, $\boz$ is a $W^{1, p}$-extension domain if and only if $\boz'$ is a $W^{1 ,p}$-extension domain. \end{theorem}
Since $M^{1, 1}(\boz)$ is strictly included in $W^{1, 1}(\boz)$  for arbitrary domains $\boz$, the above method does not work for the case $p=1$. In the same paper, Haj\l{}asz, Koskela and Tuominen raised the question.
\begin{question}\label{equs:HKT}
Is Theorem \ref{thm:HKT} true for $p=1$?
\end{question}

A partial affirmative answer to this question was provided in \cite[Corollary 1.3]{KMS}. There, it was shown that Theorem \ref{thm:HKT} holds for bounded simply connected planar domains also in the case $p=1$, and also for $BV$-functions. In \cite{KMS} it was also conjectured that the hypothesis of simply connectivity was superfluous. In this paper we will show that they were right. We will extend these previous results and answer Question \ref{equs:HKT} positively for general bounded planar domains. We will do this by first resolving the question for $BV$-functions via decomposition of sets of finite perimeter into Jordan domains, and then employing a recent result from \cite{GBR2021} to pass to $W^{1,1}$-functions.
Both, the proof in \cite{KMS} and our proof, rely on the results of V\"ais\"al\"a \cite{V2008} and on the quasiconvexity of the connected open components of the complement of planar $BV$-extension domains. The difference is that in \cite{KMS} the bi-Lipschitz function was extended to the complement by using the fact that the bi-Lipschitz map can be extended to a small neighbourhood. Here we use a decomposition of sets of finite perimeter and the bi-Lipschitz invariance of the quasiconvexity of the holes, see Section \ref{sec:proof} for the definitions and results needed for this approach.


Recall that the space of functions of bounded variation $BV(\boz)$ is defined by setting 
\[BV(\boz):=\{u\in L^1(\boz)\ :\|Du\|(\boz)<\fz \}\]
where 
\[\|Du\|(\boz)=\sup\lf\{\int_\boz u {\rm div}(v)dx\,:\,v\in C_0^\fz(\boz;\rn), |v|\leq 1\r\}\]
means the total variation of $u$ on $\boz$. The function space $BV(\boz)$ is equipped with the norm 
\[\|u\|_{BV(\boz)}:=\|u\|_{L^1(\boz)}+\|Du\|(\boz).\]
Note that $\|Du\|$ is a Radon measure on $\boz$ that is defined for every set $F\subset\boz$ as 
\[\|Du\|(F):=\inf\{\|Du\|(U)\,:\, F\subset U, U\ {\rm is\ open}\}.\]

A domain $\boz\subset\rn$ is said to be a $BV$-extension domain, if there exists a bounded extension operator $T\colon BV(\boz)\to BV(\rn)$ with $T(u)\big|_\boz\equiv u$ and an absolute constant $C>0$ so that
\[\|T(u)\|_{BV(\rn)}\leq C\|u\|_{BV(\boz)}\]
for every $u\in BV(\boz)$. By a result in \cite{KMS}, a $W^{1, 1}$-extension domain is also a $BV$-extension domain. A typical example showing that the converse is not true is the slit disk in the plane. 

In Section \ref{sec:proof} we will prove the following result.

\begin{theorem}\label{thm:biLip}
Let $\Omega \subset \mathbb R^2$ be a bounded $BV$-extension domain and $f \colon \Omega \to \Omega'$ a bi-Lipschitz map. Then  $\Omega'$ is also a $BV$-extension domain.
\end{theorem}

Let us show how Theorem \ref{thm:biLip} implies the same result for $W^{1,1}$-extension domains.
We use the recent characterization of $W^{1,1}$-extension domains among bounded $BV$-extension domains that was proven by the first and second named authors in \cite{GBR2021}.

\begin{theorem}\label{thm:planar}
Let $\Omega \subset \mathbb R^2$ be a bounded $BV$-extension domain. Then $\Omega$ is a $W^{1,1}$-extension domain if and only if
the set
\[
 \partial \Omega \setminus \bigcup_{i \in I} \overline{\Omega_i}
\]
is purely $1$-unrectifiable, where $\{\Omega_i\}_{i\in I}$ are the connected components of $\mathbb R^2 \setminus \overline{\Omega}$.
\end{theorem}

Recall that a  set $H\subset \R^2$ is called  purely $1$-unrectifiable if for every Lipschitz map $f\colon \R\to\R^2$ we have $\mathcal{H}^1(H\cap f(\R))=0$.

Suppose that $\Omega \subset \mathbb R^2$ is a bounded $W^{1,1}$-extension domain. By Theorem \ref{thm:planar} the set
  \[
 H =  \partial \Omega \setminus \bigcup_{i \in I} \overline{\Omega_i}
 \]
 is purely $1$-unrectifiable and so is the image $H' = f(H)$ under a bi-Lipschitz map $f\colon \Omega\to\Omega'$ that is extended to the closures  $\overline{\Omega}$ and $\overline{\Omega}'$ as a bi-Lipschitz map. Hence, recalling that as a $W^{1,1}$-extension domain $\Omega$ is also a $BV$-extension domain, Theorem \ref{thm:biLip} implies that $\Omega' = f(\Omega)$ is also a $BV$-extension domain. Now, from Theorem \ref{thm:planar} we conclude that $\Omega'$ is a $W^{1,1}$-extension domain. We have then established the following.
 
\begin{corollary}\label{cor:W11}
Let $\Omega \subset \mathbb R^2$ be a bounded $W^{1,1}$-extension domain and $f \colon \Omega \to \Omega'$ a bi-Lipschitz map. Then $\Omega'$ is a $W^{1,1}$-extension domain.
\end{corollary}

Let us remark that one could also prove Corollary \ref{cor:W11} with a similar proof as we provide for Theorem \ref{thm:biLip} in Section \ref{sec:proof}; via the invariance of quasiconvexity of the components of the complement, a characterization of $W^{1,1}$-extension domains as the domains with the strong extension property for sets of finite perimeter and by slightly pushing the boundary of the extension of a Jordan domain away from the boundary of $\partial \Omega$, see \cite{GBR2021} for more details on these tools. This alternative approach indicates that if we were able to prove the bi-Lipschitz invariance of $BV$-extension domains in higher dimensions, and were able to push the boundaries of sets of finite perimeter away from $\partial \Omega$ in a controlled manner, the bi-Lipschitz invariance of $W^{1,1}$-extension domains would follow. However, at the moment we are not able to complete such proof. 
An alternative approach for trying to solve the higher dimensional case could be to use the characterization of $W^{1,1}$-functions from \cite{hajlasz2}, similar to $M^{1,1}$.

\section{Proof of Theorem \ref{thm:biLip}}\label{sec:proof}

We will prove Theorem \ref{thm:biLip} by using the bi-Lipschitz invariance of the quasiconvexity of the connected components of $\mathbb R^2 \setminus \overline\Omega$, which are referred to as the \emph{holes} of $\overline{\Omega}$. 
We recall the following result of V\"ais\"al\"a \cite[Corollary 4.11]{V2008}.
\begin{theorem}\label{thm:Vaisala}
Let $G \subset \mathbb R^2$ be a bounded continuum such that each hole of $G$ is $c$-quasiconvex and let $f \colon G \to \mathbb R^2$ be $L$-bilipschitz. Then each hole of  $G' = f(G)$ is
$c'$-quasiconvex with $c'(c, L)$. 
\end{theorem}

In order to use the Theorem \ref{thm:Vaisala} we need to observe that the holes of a bounded planar $BV$-extension domain are quasiconvex. This was established for simply connected domains in  \cite[Theorem 1.1]{KMS}, and the proof works with minor modifications in the more general case considered here, see \cite[Lemma 5.2]{GBR2021}.

\begin{lemma}\label{lem:quasi.comp.domains} 
Suppose that $\Omega \subset \mathbb R^2$ is a bounded $BV$-extension domain. Then there exists a constant $c>0$ so that 
each hole of $\overline\Omega$ is $c$-quasiconvex.
\end{lemma}


The invariance of quasiconvexity of the holes is easier to use for the boundaries of sets of finite perimeter, rather than for $BV$-functions. The passage from $BV$-functions to sets of finite perimeter is provided by Lemma \ref{lma:BV-Per}, which is a combination of the works \cite{BM,KMS}. Before stating it we need to recall some definitions. 

A Lebesgue measurable subset $E\subset \R^n$ has finite perimeter in $\Omega$ if $\chi_E\in BV(\Omega)$, where $\chi_E$ denotes the characteristic function. We set $P(E,\Omega)=\|D \chi_F\|(\Omega)$ and call it the perimeter of $E$ in $\Omega$. We will say that $\Omega$ has the extension property for sets of finite perimeter if there exists $C>0$ so that for every set $E\subset \Omega$ of finite perimeter in $\Omega$ one may find $\widetilde E\subset \R^n$ of finite perimeter in $\R^n$ such that $\widetilde E\cap \Omega=E$, modulo a measure zero set, and $P(\widetilde E,\R^n)\leq CP(E,\Omega)$.

\begin{lemma}\label{lma:BV-Per}
Let $\Omega\subset\R^n$ be a bounded domain. Then the following are equivalent:
\begin{enumerate}
    \item $\Omega$ is a $BV$-extension domain.
    \item $\Omega$ has the extension property for sets of finite perimeter.
 \end{enumerate}
\end{lemma}

One more tool that we use is the decomposition of planar sets of finite perimeter into Jordan domains.

We say that $\Gamma\subset\R^2$ is a Jordan curve if $\Gamma=\gamma([a,b])$ for some $a,b\in \R$, $a<b$, and some continuous map $\gamma$, injective on $[a,b)$ and such that $\gamma(a)=\gamma(b)$. The Jordan curve theorem assures that $\Gamma$ splits $\R^2\setminus \Gamma$ into exactly two connected components, a bounded one and an unbounded one that we denote by $\text{int} (\Gamma)$ and $\text{ext} (\Gamma)$ respectively. A set $U$ whose boundary $\partial U$ is a Jordan curve is called a Jordan domain.

For technical reasons we also add to the class of Jordan curves the formal ``Jordan'' curves $J_0$  and $J_{\infty}$, whose interiors are $\R^2$ and the empty set respectively and for which we set $\mathcal{H}^{1}(J_{0})=\mathcal{H}^{1}(J_{\infty})=0$.

For a measurable set $E\subset\rn$, we denote by $\partial^M E$ its essential boundary, which consists of points such that both $E$ and $\rn\setminus E$ have positive upper density on them, that is
$$\partial^M E=\left\lbrace x\in\R^n\,:\, \limsup_{r\searrow 0}\frac{|E\cap B(x,r)|}{|B(x,r)|}>0 \;\text{and} \; \limsup_{r\searrow 0}\frac{|(\R^n\setminus E)\cap B(x,r)|}{|B(x,r)|}>0\right\rbrace .$$
A set $E \subset \mathbb R^n$ of finite perimeter is called decomposable, if there exist sets $A,B \subset \mathbb R^n$ of positive Lebesgue measure such that  $E = A \cup B$, $A \cap B = \emptyset$, and $P(E,\mathbb R^n) = P(A,\mathbb R^n) + P(B,\mathbb R^n)$. A set is called indecomposable, if it is not decomposable.

The following was proven in \cite[Corollary 1]{ACMM2001}. 

\begin{theorem} \label{thm:planardecomposition}
Let $E \subset \mathbb R^2$ have finite perimeter. Then, there exists a unique decomposition of $\partial^ME$ into rectifiable Jordan curves $\{C_i^+, C_k^-\,:\,i,k \in \mathbb N\}$, modulo $\mathcal H^1$-measure zero sets,
such that
\begin{enumerate}
    \item Given $\text{int}(C_i^+)$, $\text{int}(C_k^+)$, $i \ne k$, they are either disjoint or one is contained
in the other; given $\text{int}(C_i^-)$, $\text{int}(C_k^-)$, $i \ne k$, they are either disjoint or one is
contained in the other. Each $\text{int}(C_i^-)$ is contained in one of the $\text{int}(C_k^+)$.
    \item $P(E,\R^2) = \sum_{i}\mathcal H^1(C_i^+) + \sum_k \mathcal H^1(C_k^-)$.
    \item If $\text{int}(C_i^+) \subset \text{int}(C_j^+)$, $i \ne j$, then there is some rectifiable Jordan curve  $C_k^-$ such that $\text{int}(C_i^+)\subset \text{int}(C_k^-) \subset \text{int}(C_j^+)$. Similarly, if $\text{int}(C_i^-) \subset \text{int}(C_j^-)$, $i \ne j$, then there is some rectifiable Jordan curve  $C_k^+$ such that $\text{int}(C_i^-)\subset \text{int}(C_k^+) \subset \text{int}(C_j^-)$.
    \item Setting $L_j =\{i \,:\, \text{int}(C_i^-)\subset \text{int}(C_j^+)\}$ the sets $Y_j = \text{int}(C_j^+) \setminus \bigcup_{i \in L_j}\text{int}(C_i^-)$ are pairwise disjoint , indecomposable 
    and $E = \bigcup_j Y_j$.
\end{enumerate}
\end{theorem}
Since sets of finite perimeter are defined via the total variation seminorm of $BV$-functions, they are understood modulo $2$-dimensional measure zero sets. In particular, the last equality in (4) of Theorem \ref{thm:planardecomposition} is modulo measure zero sets. To make precise the change of representatives, we use below the notation $A \Delta B := (A\setminus B)\cup (B \setminus A)$ for the symmetric difference between subsets $A, B \subset \mathbb R^2$.

With the auxiliary tools now recalled, we are ready to prove the main result of this paper.

\begin{proof}[Proof of Theorem \ref{thm:biLip}]
  Suppose $f\colon \Omega \to \Omega'$ is $L$-bi-Lipschitz.
  First notice that $f$ extends to $\overline{\Omega} \to \overline{\Omega'}$ as a bi-Lipschitz map.
 Let $\Omega_i'$ be a connected component of  $\mathbb R^2 \setminus \overline{\Omega'}$ and $\Omega_i$ the connected component of $\mathbb R^2 \setminus \overline\Omega$ for which $f(\partial\Omega_i) = \partial\Omega_i'$. Then by Lemma \ref{lem:quasi.comp.domains} there exists a constant $c>0$ so that each hole ${\Omega_i}$ is $c$-quasiconvex. Therefore, by Theorem \ref{thm:Vaisala} also each ${\Omega_i'}$ is $c'$-quasiconvex, where the quasiconvexity constant $c'$ does not depend on $i$. Obviously, each $\overline{\Omega_i'}$ is also $c'$-quasiconvex.
 
 Suppose that $\Omega$ is a bounded $BV$-extension domain. By Lemma \ref{lma:BV-Per}, we only need to prove that having the extension property for sets of finite perimeter is invariant under the bi-Lipschitz map $f$.
 Let $E' \subset \Omega'$ be a set of finite perimeter. 
 Then
 $E = f^{-1}(E')$ is also a set of finite perimeter with
 \begin{equation}\label{eq:perimeterchange}
 P(E,\Omega) \le L P(E',\Omega').
 \end{equation}

 Let $\widetilde E$ be the perimeter extension of $E$ to the whole $\mathbb R^2$ with 
 \begin{equation}\label{eq:Omegaextension}
  P(\widetilde E, \rr^2)\leq CP(E, \boz).
 \end{equation}
 By Theorem \ref{thm:planardecomposition}, there exists a class of Jordan curves $\{\widetilde {C}_i^+,\widetilde{C}_k^-\}_{i, k\in\mathbb N}$ with 
 \begin{equation*}
 \mathcal H^2(F) = 0
 \end{equation*}
 for the symmetric difference
 \[
 F := \widetilde E \Delta \left(\bigcup_i \text{int}(\widetilde{C}_i^+) \setminus \bigcup_k 
 \text{int}(\widetilde{C}_k^-)\right),
 \]
 and
 \begin{equation}\label{eq:length}
 P(\widetilde E, \mathbb R^2)=\sum_i\mathcal H^1(\widetilde{C}_i^+)+\sum_k\mathcal H^1(\widetilde{C}^-_k)    
 \end{equation}

 Take 
 $J \in \{\text{int}(\widetilde{C}_i^+)\}_i\cup \{\text{int}(\widetilde {C}_k^-)\}_k.$  We will extend each $J' = f(J \cap \Omega) \subset \Omega'$ to the whole $\mathbb R^2$ in order to define the final extension set $\widetilde E'$ of $E'$. Consider a homeomorphism $\gamma\colon \mathbb S^1 \to \partial J$, given by the fact that $J$ is a Jordan domain. The set $\gamma\setminus \overline{\Omega}$ consists of (at most) countably many open arcs $\gamma_i$ with endpoints $x_i,y_i \in \partial\Omega_{j(i)}$ for some $j(i)$. Observe that since $|x_i-y_i| \le \mathcal H^1(\gamma_i)$,
 we always have
 \begin{equation}\label{eq:onecurve}
 \mathcal{H}^1(\gamma\cap\overline{\Omega})+ \sum_i |x_i- y_i|\leq  \mathcal{H}^1(\gamma\cap\overline{\Omega})+\mathcal{H}^1(\gamma\setminus \overline\Omega)\leq \mathcal{H}^1(\gamma).
 \end{equation}
  Let us write $I_i=\gamma^{-1}(\gamma_i)$ for each $i$ 
  and use the  $c'$-quasiconvexity of  $\overline{\Omega_{j(i)}'}$ to find a curve $\gamma_i' \subset \overline{\Omega_{j(i)}'}$ joining $f(x_i)$ to $f(y_i)$ with
 \begin{equation}\label{eq:gammaiqc}
 \ell(\gamma_i') \le c' |f(x_i)- f(y_i)|.
 \end{equation}
 For convenience, we use the parametrization $\gamma_i' \colon I_i \to \mathbb R^2$ so that $\gamma_i'^{-1}(x_i) = \gamma^{-1}(x_i)$
 and $\gamma_i'^{-1}(y_i) = \gamma^{-1}(y_i)$.
 The combination of $f(\gamma \cap \overline{\Omega})$
 with the curves $\gamma_i'$ results 
 in a continuous curve $\gamma' \colon \mathbb S^1 \to \mathbb R^2$
 defined as
 \[
 \gamma'(t) = \begin{cases}
  f(\gamma(t)), &\text{if }\gamma(t) \in \overline{\Omega},\\
  \gamma_i'(t), & \text{if }t \in I_i.
\end{cases}
 \]
  Now, we define the extension domain $\widetilde J'$ of $J'=f(\Omega\cap J)$ as the union of all the connected components of $\mathbb R^2 \setminus \gamma'$ that intersect $J'$. 
 Then, combining \eqref{eq:gammaiqc} with \eqref{eq:onecurve} and the fact that $f$ is $L$-bi-Lipschitz, gives
 \begin{equation}\label{eq:pieces}
 \begin{split}
 P(\widetilde J',\mathbb R^2) & \le \mathcal{H}^1(\gamma') \\
  & = \mathcal H^1(f(\gamma \cap \overline{\Omega}))
   + \sum_i\mathcal H^1(\gamma_i') \\
 & \le \mathcal H^1(f(\gamma \cap \overline{\Omega}))
   + \sum_i c'|f(x_i)-f(y_i)|\\
   & \le L \mathcal H^1(\gamma \cap \overline{\Omega}) + 
   c'L \sum_i\mathcal|x_i-y_i| \\
   &\le  c'L \mathcal H^1(\gamma) \\
   &=c'L P(J,\mathbb R^2).
  \end{split}
  \end{equation}
 
We finally set our extension of $E'$ to be
$$ \widetilde E'=\bigcup_{J\in \{\text{int}(\widetilde{C}_i^+)\}_i}  \widetilde J' \setminus\bigcup_{J\in \{\text{int}(\widetilde {C}_k^-)\}_k} \widetilde J'.$$
  Since the decomposition of $\widetilde E$ was left unchanged inside $\Omega$, we have that 
  \[
  (\widetilde E' \cap \Omega') \Delta E' = f(F\cap \Omega)
  \]
  has zero $2$-dimensional measure, since $F$ is measure-zero and $f$ is bi-Lipschitz. Hence, $\widetilde{E'}$ is indeed an extension of $E'$
   
 Now, summing the estimate \eqref{eq:pieces} over all the Jordan curves and using \eqref{eq:length}, \eqref{eq:Omegaextension}, and
 \eqref{eq:perimeterchange}, we get
 \begin{align*}
  P(\widetilde E',\R^2)&\leq \sum_{J\in \{\text{int}(\widetilde{C}_i^+),\text{int}(\widetilde {C}_k^-)\}_{i,k}} P(\widetilde{J'},\R^2)\\
  &\le c'L \sum_{J\in \{\text{int}(\widetilde{C}_i^+),\text{int}(\widetilde {C}_k^-)\}_{i,k}} P(J,\R^2)\\
  &= c'L P(\widetilde E,\mathbb R^2) \\
  &\le Cc'L P(E,\Omega)\\
  &\le Cc'L^2 P(E',\Omega').
 \end{align*}
This shows that $\widetilde E'$ is a perimeter extension of $E$ as required. Thus, by Lemma \ref{lma:BV-Per} we conclude that the domain $\Omega'$ is a $BV$-extension domain.
\end{proof}

\end{document}